\documentclass[10pt,twoside]{article}
\usepackage{graphics}
\usepackage{graphicx}
\usepackage{epstopdf}
\usepackage{amsfonts,amsmath,amsthm,amssymb,bm,hyperref}

\newtheorem{proposition}{Proposition}
\newtheorem{definition}{Definition}
\newtheorem{remark}{Remark}
\newtheorem{theorem}{Theorem}
\newtheorem{lemma}{Lemma}

\newcommand{\ds}{\displaystyle}

\begin{document}

 \title{Analysis of two- and three-dimensional fractional-order Hindmarsh-Rose type neuronal models}

 \author{Eva Kaslik\\
$^1$ Institute e-Austria Timisoara, Bd. V. Parvan nr. 4, room 045B, 300223, Timisoara, Romania\\
$^2$ Dept. of Mathematics and Computer Science, West University of Timisoara, Romania\\
  e-mail: ekaslik@gmail.com
}

\date{}
\maketitle

{\large\textbf{PUBLICATION DETAILS:\\
This paper is now published (in revised form) in\\ Fractional Calculus and Applied Analysis, 20(3): 623-–645, 2017,\\
\href{http://doi.org/10.1515/fca-2017-0033}{DOI:10.1515/fca-2017-0033},\\
and is available online at \url{http://www.degruyter.com/view/j/fca}.}}

\section*{\small{Abstract}}
\begin{small}

A theoretical analysis of two- and three-dimensional fractional-order Hindmarsh-Rose neuronal models is presented, focusing on stability properties and occurrence of Hopf bifurcations, with respect to the fractional order of the system chosen as bifurcation parameter. With the aim of exemplifying and validating the theoretical results, numerical simulations are also undertaken, which reveal rich bursting behavior in the three-dimensional fractional-order slow-fast system.

\medskip

{\it MSC 2010\/}: Primary 26A33;
                  Secondary 33E12, 34A08, 34K37, 35R11, 60G22.

 \smallskip

{\it Key Words and Phrases}: fractional-order, Hindmarsh-Rose, neuron, neuronal activity, stability, Hopf bifurcation, bursting, slow-fast system.
\end{small}

\section{Introduction}

Neuronal activity of biological neurons is typically modeled using the classical Hodgkin-Huxley mathematical model \cite{Hodgkin-Huxley}, dating back to 1952, that includes nonlinear differential equations for the membrane potential and gating variables of ionic currents. Simplified versions of the Hodgkin-Huxley model have been introduced in 1962 by Fitzhugh and Nagumo \cite{FitzHugh} and in 1981 by Morris and Lecar \cite{Morris_Lecar_1981}.

In 1982, Hindmarsh and Rose \cite{Hindmarsh-Rose-1} introduced a different simplification of the original Hodgkin-Huxley model, proposing the following two-dimensional model of neuronal activity:
\begin{equation}\label{eq.HR.2d}
\left\{
  \begin{array}{ll}
    \dot{x} = y-ax^3+bx^2+I \\
    \dot{y} = c-dx^2-y
  \end{array}
\right.
\end{equation}
where $x$ represents the membrane potential in the axon of a neuron and $y$ is a recovery variable, called the \emph{spiking variable}, representing the transport rate of sodium and potassium ions through fast ion channels. The parameters $a$, $b$, $c$ and $d$ are positive ($a=c=1$ is often assumed) and $I$ represents the external stimulus.

Two years later, Hindmarsh and Rose  \cite{Hindmarsh-Rose-2} decided to improve their model by adding a third equation, that takes into account a slow adaptation current $z$. The three-dimensional Hindmarsh-Rose model is described by the following system of three differential equations:
\begin{equation}\label{eq.HR.3d}
\left\{
  \begin{array}{ll}
    \dot{x} = y-ax^3+bx^2+I-z \\
    \dot{y} = c-dx^2-y\\
    \dot{z} =\varepsilon (s(x-x_0)-z)
  \end{array}
\right.
\end{equation}
where $(x_0,y_0)$ are the coordinates of the leftmost equilibrium point of the system without adaptation (\ref{eq.HR.2d}). Here, the variable $z$, called the \emph{bursting variable}, represents the exchange of ions through slow ionic channels. The parameters $\varepsilon$ and $s$ are positive, and $\varepsilon$ is considered to be small.

It has been previously noted that this extra mathematical complexity allows a great variety of dynamic behaviors for the membrane potential $x$, including chaotic dynamics. Therefore, the Hindmarsh-Rose neuron model has a great importance: while still being relatively simple, it allows for a good qualitative description of many different patterns of the action potential observed in experiments.

An important phenomenon in neuron activity is the transition between \emph{spiking}, represented by a generation of action potentials, and \emph{bursting}, represented by a membrane potential changing from resting to repetitive firing state. Bifurcation phenomena correspond to qualitative changes of the information transmitted through the axon of the neuron, determining the transition between a quiescent state and an oscillatory one, or between different kinds of oscillatory behaviors. Hence, bifurcation theory plays an important role in studying the dynamics of system (\ref{eq.HR.3d}). We refer to \cite{Corson_2009,Innocenti_2007,Innocenti_2009,Shilnikov_2008,Storace_2008,Tsuji_2007} for recent results concerning the dynamics of integer-order Hindmarsh-Rose models.

System (\ref{eq.HR.3d}) can be regarded as a slow-fast system, where the fast subsystem is given by (\ref{eq.HR.2d}), whose bifurcation diagram  provides important information about the dynamics and bursting patterns of system (\ref{eq.HR.3d}) when $\varepsilon$ is small enough \cite{Guckenheimer-2011}. Indeed, based on the method of dissection of neuronal bursting \cite{Rinzel}, setting $\varepsilon=0$ in (\ref{eq.HR.3d}) and studying the fast subsystem by treating $z$ as a bifurcation parameter, typically, the fast subsystem exhibits a limit cycle for some values of $z$ and an equilibrium point for other values of $z$. Therefore, as the slow variable $z$ in system (\ref{eq.HR.3d}) oscillates between two values, the whole system will burst.

In this paper, improved versions of the two- and three-dimensional Hindmarsh-Rose models are proposed and analyzed, by replacing the integer-order derivatives by fractional-order Caputo-type derivatives \cite{Kilbas,Podlubny,Lak}. This fractional-order formulation is justified by research results concerning biological neurons. Indeed, the results reported in \cite{Anastasio} suggest that "the oculomotor integrator, which converts eye velocity into eye position commands, may be of fractional order. This order is less than one, and the velocity commands have order one or greater, so the resulting net output of motor and premotor neurons can be described as fractional differentiation relative to eye position". Moreover, in the recent paper \cite{Lundstrom} it has been pointed out that "fractional differentiation provides neurons with a fundamental and general computation ability that can contribute to efficient information processing, stimulus anticipation and frequency-independent phase shifts of oscillatory neuronal firing", emphasizing once again the utility of developing and studying fractional-order mathematical models of neuronal activity.

The main benefit of fractional-order models in comparison with classical integer-order models is that fractional derivatives provide a good tool for the description of memory and hereditary properties of various processes. This is obviously a desired feature when it comes to the modelling of a biological neuron. In fact, fractional-order systems are characterized by infinite memory, as opposed to integer-order systems. The generalization of dynamical equations using fractional derivatives proved to be more accurate in the mathematical modeling of real world phenomena arising from several interdisciplinary areas, such as phenomenological description of viscoelastic liquids \cite{Heymans_Bauwens}, diffusion and wave propagation \cite{Henry_Wearne,Metzler}, colored noise \cite{Cottone}, boundary layer effects in ducts \cite{Sugimoto}, electromagnetic waves \cite{Engheia}, fractional kinetics \cite{Mainardi_1996}, electrode-electrolyte polarization \cite{Ichise_Nagayanagi_Kojima}, etc.

Fractional-order models of Hindmarsh-Rose type have been recently studied in \cite{Min_2012,Min_2013,Jun_2014,Xie_2014}. In \cite{Min_2012,Min_2013}, the author obtains stability and bifurcation results for a two-dimensional modified fractional-order Hindmarsh-Rose neuronal model, and introduces a state feedback method to control the Hopf bifurcation. In \cite{Jun_2014,Xie_2014}, a three-dimensional fractional-order Hindmarsh-Rose model is considered with fixed numerical values of the system parameters, and extensive numerical simulations are carried out to exemplify the dynamical characteristics of the model, without focusing on theoretical analysis. Recently, a fractional-order Morris-Lecar neuron model with fast-slow variables has been investigated in \cite{Shi_2014}, revealing some bursting patterns that do not exist in the corresponding integer-order model.

This paper is devoted to the theoretical analysis of the two- and three-dimensional fractional-order Hindmarsh-Rose neuronal models, focusing on stability properties and occurrence of Hopf bifurcations, choosing the fractional order of the system as bifurcation parameter. The theoretical results are obtained in a general framework, without specifying the numerical values of the system parameters, which is often the case in previously published papers \cite{Jun_2014,Xie_2014}. Numerical simulations are also undertaken, with the aim of exemplifying the theoretical results and revealing bursting behaviour in the three-dimensional model.

\section{Preliminaries on fractional-order differential systems}

In general, three different definitions of fractional derivatives are widely used: the Gr\"{u}nwald-Letnikov derivative, the Riemann-Liouville derivative and the Caputo derivative. These three definitions are in general non-equivalent. However, the main advantage of the Caputo derivative is that it only requires initial conditions given in terms of integer-order derivatives, representing well-understood features of physical situations and thus making it more applicable to real world problems.

\begin{definition}
For a continuous function $f$, with $f'\in L^1_{loc}(\mathbb{R}^+)$, the Caputo fractional-order derivative of order $q\in(0,1)$ of $f$ is defined by
$$
^cD^q f(t)=\frac{1}{\Gamma(1-q)}\int_
0^t(t-s)^{-q}f'(s)ds.
$$
\end{definition}

\begin{remark}
When $q\rightarrow 1$, the fractional order derivative $^cD^q f(t)$ converges to the integer-order derivative $f'(t)$.
\end{remark}

Highly remarkable scientific books which provide the main theoretical tools for the qualitative analysis of fractional-order dynamical systems, and at the same time, show the interconnection as well as the contrast between classical differential equations and fractional differential equations, are \cite{Podlubny,Kilbas,Lak}.

An analogue of the classical Hartman theorem for nonlinear integer-order dynamical systems, the linearization theorem for fractional-order dynamical systems has been recently proved in \cite{Li_Ma_2013}. The following stability result holds for linear autonomous fractional-order systems \cite{Matignon}:

\begin{theorem}
\label{thm.linear.stab}
The linear fractional-order autonomous system
$$^cD^q \bold{x}=A \bold{x}\qquad\textrm{where}~~ A\in\mathbb{R}^{n\times n}$$
where $q\in(0,1)$ is asymptotically stable if and only if
\begin{equation}
\label{eq.lambda.spec}
|\arg(\lambda)|>\frac{q\pi}{2}\qquad\forall \lambda\in\sigma(A)
\end{equation}
where $\sigma(A)$ denotes the spectrum of the matrix $A$ (i.e. the set of all eigenvalues).
\end{theorem}

The following result can be easily shown using basic mathematical tools.

\begin{lemma}
\label{lem.lambda}
Let $q\in(0,1)$. The complex number $\lambda$ satisfies
$$|\arg(\lambda)|>\frac{q\pi}{2}$$ if and only if one of the following hold:
\begin{itemize}
  \item[(i)] $\Re(\lambda)<0$;
  \item[(ii)] $\Re(\lambda)\geq 0$ and $ |\Im(\lambda)|>\Re(\lambda)\tan\frac{q\pi}{2}$.
\end{itemize}
\end{lemma}

\begin{remark}
\label{rem.stab.comp}
For the integer order system $\dot{\bold{x}}=A \bold{x}$, the null solution is asymptotically stable if and only if condition (i) from Lemma \ref{lem.lambda} is satisfied for any $\lambda\in\sigma(A)$. Lemma \ref{lem.lambda} shows that in the case of linear fractional-order systems, the conditions for the asymptotic stability of the null solution become more relaxed than in the integer-order case, due to the alternative provided by (ii). It is also worth noting that if the matrix $A$ does not have positive real eigenvalues, it is possible to choose a fractional order $q$ such that the null solution of $^cD^q \bold{x}=A \bold{x}$ is asymptotically stable. The existence of at least one positive real root of $A$ guarantees the instability of the null solution for any fractional order $q\in(0,1]$.
\end{remark}

In the following, a general result for the stability of a two-dimensional fractional order dynamical system will be explored, using Lemma \ref{lem.lambda}.

\begin{proposition}
\label{prop.stability.2d}
Let $q\in(0,1)$. The two-dimensional linear fractional-order system
$$^cD^q \bold{x}=A \bold{x}\qquad\textrm{where}~~ A\in\mathbb{R}^{2\times 2}$$ is asymptotically stable if and only if one of the following conditions hold:
\begin{itemize}
  \item[(i)] $\tau<0$ and $\delta>0$;
  \item[(ii)] $\tau\geq 0$ and $\ds\delta>\frac{\tau^2}{4}\sec^2{\left(\frac{q\pi}{2}\right)}$;
\end{itemize}
where $\tau=\textrm{trace}(A)$ and $\delta=\det(A)$.
\end{proposition}

\begin{proof}
The eigenvalues $\lambda_{1,2}$ of the matrix $A$ satisfy the characteristic equation:
$$\lambda^2-\tau\lambda+\delta=0.$$
We have the following cases:

\noindent 1. $\tau^2-4\delta>0$. In this case, the eigenvalues are $$\lambda_{\pm}=\frac{1}{2}(\tau\pm\sqrt{\tau^2-4\delta})\in\mathbb{R}.$$ Since $\Im(\lambda_+)=\Im(\lambda_-)=0$, it can be easily seen that condition (ii) from Lemma \ref{lem.lambda} cannot be satisfied. However, condition (i) from Lemma \ref{lem.lambda} is equivalent in this case to $\lambda_+<0$, or equivalently
$$\tau<0\quad\textrm{and}\quad0<\delta\leq \frac{\tau^2}{4}.$$

\noindent 2. $\tau^2-4\delta\leq 0$. In this case, the eigenvalues are $$\lambda_{\pm}=\frac{1}{2}(\tau\pm i\sqrt{4\delta-\tau^2})\in\mathbb{C}\setminus\mathbb{R}.$$
In this case, Condition (i) from Lemma \ref{lem.lambda} is equivalent to $$\tau<0\quad\textrm{and}\quad\delta> \frac{\tau^2}{4}.$$
On the other hand, condition (ii) from Lemma \ref{lem.lambda} is equivalent to
$$\tau\geq 0 \quad\textrm{and}\quad \sqrt{4\delta-\tau^2}>\tau\tan\frac{q\pi}{2},$$
or equivalently
$$\tau\geq 0 \quad\textrm{and}\quad 4\delta>\tau^2\sec^2\left(\frac{q\pi}{2}\right).$$
\end{proof}

\begin{remark}\label{rem.equiv}
The two conditions (i) and (ii) from Proposition \ref{prop.stability.2d} can be replaced by the equivalent condition:
$$\delta>0\qquad\textrm{and}\qquad \frac{\tau}{\sqrt{\delta}}<2\cos\left(\frac{q\pi}{2}\right).$$
\end{remark}

\section{The two-dimensional fractional-order Hindmarsh-Rose model}

\subsection{Model description and equilibrium states}

In this section, we will focus on the fast subsystem, considering the following two-dimensional fractional-order Hindmarsh-Rose model:
\begin{equation}\label{sys.2DFHR}
\left\{
  \begin{array}{ll}
    ^cD^q x = y-F(x)+I \\
    ^cD^q y = G(x)-y
  \end{array}
\right.
\end{equation}
where $x$ denotes the cell membrane potential and $y$ represents a recovery variable, while $I$ denotes the external stimulus. The fractional order of the system (\ref{sys.2DFHR}) is $q\in(0,1)$. The following assumptions are considered for the functions $F$ and $G$ (see \cite{Hindmarsh-Rose-2}):
\begin{itemize}
\item[i.] $F$ is cubic and $F(x)\rightarrow\infty$ as $x\rightarrow\infty$;
\item[ii.] $G$ is quadratic;
\item[iii.] both $F$ and $G$ have a local maximum value at $x=0$ and $F(0)=0$.
\end{itemize}
With these conditions, the general forms of $F$ and $G$ are
\begin{align*}
F(x)&=ax^3-bx^2\\
G(x)&=c-dx^2
\end{align*}
where $a,b,c,d$ are positive constants. Usually, it is assumed that $a=c=1$.

The equilibrium states of system (\ref{sys.2DFHR}) satisfy the following equations:
\begin{equation}\left\{
    \begin{array}{ll}
      y-F(x)+I=0 \\
      G(x)-y=0
    \end{array}
  \right.
\end{equation}
Eliminating $y$ from this system ($y=G(x)$), we obtain $F(x)-G(y)=I$ which can be written in the form:
\begin{equation}x^3-px^2=r\end{equation}
where $\ds p=\frac{b-d}{a}$ and $\ds r=\frac{I+c}{a}$.


Denoting $h(x)=x^3-p x^2$, the roots of its derivative
$h'(x)=x(3x-2p)$ are $0$ and  $\ds\frac{2p}{3}$. We denote these roots as follows:
$$\ds\alpha_1=\min\left\{0,\frac{2p}{3}\right\}\qquad\textrm{and}\qquad\ds\alpha_2=\max\left\{0,\frac{2p}{3}\right\}.$$
Hence, the function $h$ has a local maximum at $\alpha_1$ and a local minimum at $\alpha_2$, where $h(\alpha_1)>h(\alpha_2)$. We consider the following bijective functions:
\begin{itemize}
\item $h_1:(-\infty,\alpha_1)\rightarrow (-\infty,h(\alpha_1))$, the restriction of $h$ to $(-\infty,\alpha_1)$, i.e. $h_1=h|_{(-\infty,\alpha_1)}$;
\item $h_2:[\alpha_1,\alpha_2]\rightarrow [h(\alpha_2),h(\alpha_1)]$, the restriction of $h$ to $[\alpha_1,\alpha_2]$, i.e. $h_2=h|_{[\alpha_1,\alpha_2]}$;
\item $h_3:(\alpha_2,\infty)\rightarrow (h(\alpha_2),\infty)$, the restriction of $h$ to $(\alpha_2,\infty)$, i.e. $h_3=h|_{(\alpha_2,\infty)}$;
\end{itemize}
The functions $h_1$ and $h_3$ are strictly increasing, while the function $h_2$ is decreasing.

Hence, we have three branches of equilibrium states for system (\ref{sys.2DFHR}), with respect to the parameter $r$:
\begin{itemize}
\item For $r\in(-\infty,h(\alpha_1))$, we have the branch of equilibrium states $E_1(r)=(x_1^\star(r),G(x_1^\star(r)))$, with $x_1^\star(r)=h_1^{-1}(r)<\alpha_1$.
\item For $r\in[h(\alpha_2),h(\alpha_1)]$, we have the branch of equilibrium states $E_2(r)=(x_2^\star(r),G(x_2^\star(r)))$, with $x_2^\star(r)=h_2^{-1}(r)\in[\alpha_1,\alpha_2]$.
\item For $r\in(h(\alpha_2),\infty)$, we have the branch of equilibrium states $E_3(r)=(x_3^\star(r),G(x_3^\star(r)))$, with $x_3^\star(r)=h_3^{-1}(r)>\alpha_2$.
\end{itemize}

\begin{remark}
The three branches of equilibrium states coexist if and only if $r\in(h(\alpha_2),h(\alpha_1))$, meaning that for every value of $r$ in the interval $(h(\alpha_2),h(\alpha_1))$, there are exactly three equilibrium states of system (\ref{sys.2DFHR}). On the other hand, for every $r$ outside the interval $[h(\alpha_2),h(\alpha_1)]$ there is a unique equilibrium state  of system (\ref{sys.2DFHR}).
\end{remark}

\subsection{Stability of equilibrium states}

In this section, we will analyze the stability of the equilibrium states of system (\ref{sys.2DFHR}), based on the results presented in Proposition \ref{prop.stability.2d} and Remark \ref{rem.equiv}.

The jacobian matrix of (\ref{sys.2DFHR}) at an equilibrium state $E^\star=(x^\star,G(x^\star))$ is
$$A(x^\star)=\left(
                       \begin{array}{cc}
                         -F'(x^\star) & 1 \\
                         G'(x^\star) & -1 \\
                       \end{array}
                     \right)
$$
and its trace and determinant are given by
\begin{align*}
\tau(x^\star)&=-F'(x^\star)-1=-3a(x^\star)^2+2bx^\star-1\\
\delta(x^\star)&=F'(x^\star)-G'(x^\star)=ah'(x^\star)=ax^\star(3x^\star-2p)
\end{align*}

\begin{remark}\label{rem.tau.delta}
We will now shortly analyze the sign of the trace $\tau(x)$ and the determinant $\delta(x)$.

If $b^2<3a$ then $\tau(x)$ does not have real roots and therefore $\tau(x)=-3ax^2+2bx-1<0$ for any $x\in\mathbb{R}$.

On the other hand, if $b^2\geq 3a$ then $\tau(x)=-3ax^2+2bx-1$ has positive real roots
$$\ds\gamma_1=\frac{b-\sqrt{b^2-3a}}{3a}\qquad\textrm{and}\qquad\ds\gamma_2=\frac{b+\sqrt{b^2-3a}}{3a}.$$
In this case, $\tau(x)$ will be positive for $x\in(\gamma_1,\gamma_2)$ and negative otherwise.

Since $\alpha_1$ and $\alpha_2$ are the roots of $h'$, it can be easily seen that the determinant $\delta(x)$ is positive, whenever $x<\alpha_1$ or $x>\alpha_2$, and negative if $x\in(\alpha_1,\alpha_2)$. From Proposition \ref{prop.stability.2d}, it easily follows that if $\delta(x^\star)\leq 0$, the equilibrium state $(x^\star,G(x^\star))$ cannot be asymptotically stable.
\end{remark}

\begin{proposition}\label{prop.stab.equi}
Regarding the stability of the equilibrium states, the following results holds:
\begin{itemize}
\item[(a)] The equilibrium states belonging to the first branch $E_1(r)$, where $r\in(-\infty,h(\alpha_1))$, are asymptotically stable.
\item[(b)] The equilibrium states belonging to the second branch $E_2(r)$, where $r\in(h(\alpha_2),h(\alpha_1))$, are unstable.
\item[(c.1)] If $b^2\leq 3a$, the equilibrium states belonging to the third branch $E_3(r)$, where $r>h(\alpha_2)$, are asymptotically stable.
\item[(c.2)] If $b^2>3a$, the equilibrium states belonging to the third branch $E_3(r)$, where $r>h(\alpha_2)$, are asymptotically stable if and only if one of the following cases hold:
\begin{itemize}
\item[1.]  $\ds p\leq\frac{3\gamma_1}{2}$ and $r\in (h(\alpha_2),h(\gamma_1))$;
\item[2.]  $\ds p<\frac{3\gamma_2}{2}$ and $r\in  (h(\gamma_2),\infty)$;
\item[3.] $\ds p\geq \frac{3\gamma_2}{2}$;
\item[4.]  $\ds p\leq\frac{3\gamma_1}{2}$, $r\in [h(\gamma_1),h(\gamma_2)]$ and $\ds q<\frac{2}{\pi}\arccos \left(\frac{\tau(h_3^{-1}(r))}{2\sqrt{\delta(h_3^{-1}(r))}}\right)$;
\item[5.] $\ds \frac{3\gamma_1}{2}<p<\frac{3\gamma_2}{2}$, $r\in \left(\ds-\frac{4p^3}{27},h(\gamma_2)\right]$ and $\ds q<\frac{2}{\pi}\arccos \left(\frac{\tau(h_3^{-1}(r))}{2\sqrt{\delta(h_3^{-1}(r))}}\right).$
\end{itemize}
\end{itemize}
\end{proposition}

\begin{proof}

(a) Let us first consider $r\in(-\infty, h(\alpha_1))$ and the corresponding equilibrium state $E_1(r)$ from the first branch. As $x_1^\star(r)<\alpha_1\leq 0<\gamma_1$, we have from Remark \ref{rem.tau.delta} that $\delta(x_1^\star(r))>0$ and $\tau(x_1^\star(r))<0$. Therefore, based on Proposition \ref{prop.stability.2d}, we obtain that $E_1(r)$ is asymptotically stable.

(b) Let $r\in(h(\alpha_2),h(\alpha_1))$ and the corresponding equilibrium state $E_2(r)$ from the second branch. We know that $x_2^\star(r)\in (\alpha_1,\alpha_2)$, and hence, $\delta(x_2^\star(r))<0$, meaning that $E_2(r)$ is unstable.

(c.1) Let $r\in(h(\alpha_2),\infty)$ and the corresponding equilibrium state $E_3(r)$ from the third branch. As $x_3^\star(r)=h_3^{-1}(r)>\alpha_2$, it follows that $\delta(x_3^\star(r))>0$.

If $b^2\leq 3a$, we have seen that $\tau(x)=-3ax^2+2bx-1$ is always negative, and therefore, $\tau(x_3^\star(r))\leq 0$. Therefore, based on Proposition \ref{prop.stability.2d}, we obtain that $E_3(r)$ is asymptotically stable, for any $r\in(h(\alpha_2),\infty)$.

(c.2) Let $b^2> 3a$. Similarly as in the case (c.1), it follows that $\delta(x_3^\star(r))>0$ for any $r\in(h(\alpha_2),\infty)$.

We first consider $\ds p\leq\frac{3\gamma_1}{2}$. In this case, $\ds\alpha_2=\max\{0,\frac{2p}{3}\}\leq\gamma_1<\gamma_2$. If $x_3^\star(r)=h_3^{-1}(r)<\gamma_1$, which is equivalent to $r<h(\gamma_1)$, we have $\tau(x_3^\star(r))<0$ and we obtain that $E_3(r)$ is asymptotically stable.

Considering $\ds p<\frac{3\gamma_2}{2}$, we have $\ds\alpha_2=\max\{0,\frac{2p}{3}\}< \gamma_2$.
If $x_3^\star(r)=h_3^{-1}(r)>\gamma_2$, which is equivalent to $r>h(\gamma_2)$, we have $\tau(x_3^\star(r))<0$ and we obtain that $E_3(r)$ is asymptotically stable.

If $\ds p\geq\frac{3\gamma_2}{2}$, it is obvious that $\ds\gamma_2\leq \frac{2p}{3}=\alpha_2$, and hence, for any $r>h(\alpha_2)$, we have $x_3^\star(r)=h_3^{-1}(r)>\alpha_2\geq \gamma_2$. Therefore, $\tau(x_3^\star(r))<0$ and we obtain that $E_3(r)$ is asymptotically stable.

The last two cases are actually equivalent to the fact that $x_3^\star(r)=h_3^{-1}(r)\in[\gamma_1,\gamma_2]$ and we obtain that $\tau(x_3^\star(r))\geq 0$. From Proposition \ref{prop.stability.2d} and Remark \ref{prop.stab.equi}, the equilibrium state $E_3(r)$ is asymptotically stable if and only if
$$\frac{\tau(h_3^{-1}(r))}{\sqrt{\delta(h_3^{-1}(r))}}<2\cos\left(\frac{q\pi}{2}\right)$$
which is equivalent to
$$q<\frac{2}{\pi}\arccos \left(\frac{\tau(h_3^{-1}(r))}{2\sqrt{\delta(h_3^{-1}(r))}}\right).$$
The proof is now complete.
\end{proof}

\begin{remark}
It is important to notice that if we exclude the last two cases from (c.2) in Proposition \ref{prop.stability.2d}, we obtain necessary and sufficient conditions for the asymptotic stability of the equilibrium states, regardless of the fractional order $q$. Moreover, these conditions correspond to the asymptotic stability of the equilibrium states in the framework of the classical integer order system. Conditions (c.2.4) and (c.2.5) specifically correspond to the fractional order case, and they show that it is possible to stabilize the equilibrium states from the branch $E_3(r)$ by a suitable choice of the fractional order $q$.
\end{remark}

\subsection{Remarks on Hopf bifurcation phenomena}

Concerning the bifurcation phenomena occurring in fractional-order dynamical systems, very few results are known at this moment. The recent paper \cite{ElSaka} attempts to formulate conditions for the Hopf bifurcation, based on observations arising from numerical simulations. However, the complete characterization of the Hopf bifurcation and the stability of the resulting limit cycle is still an open question.

As it can be seen from Proposition \ref{prop.stab.equi}, the two main parameters that characterize the stability and bifurcation parameters in system (\ref{sys.2DFHR}) are $r=\ds\frac{I+c}{a}$, which is determined by the external stimulus, and the fractional order $q$ of the system.

Based on \cite{ElSaka} and the asymptotic stability results presented in Proposition \ref{prop.stab.equi}, it may be concluded that Hopf bifurcations can take place in (\ref{sys.2DFHR}), only on the third branch of equilibria, when $b^2>3a$ and the fractional order $q$ reaches the critical value \begin{equation}\label{eq.q.critical.2d}\ds q^\star(r)=\frac{2}{\pi}\arccos \left(\frac{\tau(h_3^{-1}(r))}{2\sqrt{\delta(h_3^{-1}(r))}}\right),
\end{equation}
in one of the following cases
\begin{itemize}
\item $\ds p\leq \frac{3\gamma_1}{2}$ and $r\in [h(\gamma_1),h(\gamma_2)]$;
\item $\ds\frac{3\gamma_1}{2}<p<\frac{3\gamma_2}{2}$ and $r\in \left(\ds-\frac{4p^3}{27},h(\gamma_2)\right]$.
\end{itemize}
In these two cases, taking into account that $q=q^\star(r)$  we have
$$\frac{\tau(h_3^{-1}(r))}{2\sqrt{\delta(h_3^{-1}(r))}}=\cos\frac{q\pi}{2}.$$
The eigenvalues of the jacobian matrix $A(x_3^\star(r))$ are
\begin{align*}\lambda_\pm&=\frac{1}{2}\left(\tau\pm i\sqrt{4\delta-\tau^2}\right)\\
&=\frac{\tau}{2}\left(1\pm i\sqrt{\frac{4\delta}{\tau^2}-1}\right)\\
&=\frac{\tau}{2}\left(1\pm i\sqrt{\sec^2\left(\frac{q\pi}{2}\right)-1}\right)\\
&=\frac{\tau}{2}\left(1\pm i\tan\left(\frac{q\pi}{2}\right)\right)\\
&=\frac{\tau}{2}\sec\frac{q\pi}{2}\exp\left(\pm i\frac{q\pi}{2}\right)
\end{align*}
Hence, $\ds \arg(\lambda_\pm)=\pm \frac{q\pi}{2}$, and according to \cite{ElSaka}, this corresponds to a Hopf bifurcation in the fractional-order system (\ref{sys.2DFHR}).

\subsection{Numerical example}

For all numerical simulations, the generalization of the Adams-Bashforth-Moulton predictor-corrector method has been used \cite{Diethelm}. The drawback of all numerical methods available for fractional-order dynamical systems is that, in order to obtain a reliable estimation of the solution, because of the hereditary nature of the problem, at every iteration step all previous iterations have to be taken into account, and hence, the computational costs are very high if the solution is computed over a large time interval.

We will consider the following values for the parameters appearing in system (\ref{sys.2DFHR}):
$$a=1;\quad b=3;\quad c=1;\quad d=5.$$
These are the reference values given by Hindmarsh and Rose \cite{Hindmarsh-Rose-2}, and used frequently in the literature for numerical simulations.

In this case, we obviously have $b^2>3a$. We compute $\ds p=\frac{b-d}{a}=-2<0$, $\ds r=\frac{I+c}{a}=I+1$ and $\ds \gamma_{1,2}=1\pm\frac{\sqrt{6}}{3}$.

System (\ref{sys.2DFHR}) has three equilibria if and only if $r\in[h(0),h(2p/3)]=[0,1.18519]$. Outside this interval, (\ref{sys.2DFHR}) has a unique equilibrium state.

Based on the previous remarks, a Hopf bifurcation may occur in system (\ref{sys.2DFHR}), in a neighborhood of the equilibrium state $E_3(r)$, if and only if $$r\in[h(\gamma_1),h(\gamma_2)]=[0.07353,12.5931]$$
when the fractional order $q$ exceeds the critical value $q^\star(r)$ given by (\ref{eq.q.critical.2d}). The critical values $q^\star(r)$ belong to the Hopf bifurcation curve represented in Fig. \ref{fig.stabreg.2d}.

\begin{figure}[htbp]
\centering
\includegraphics*[width=0.5\linewidth]{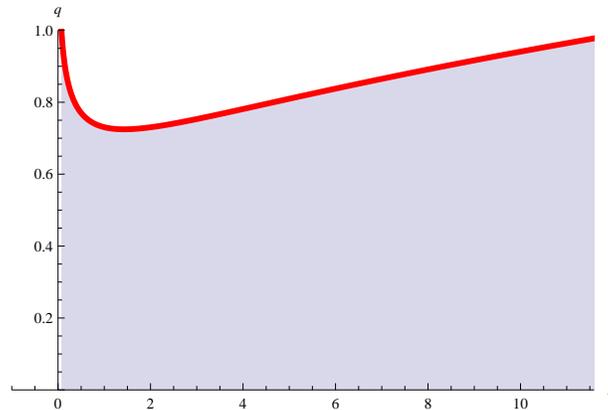}
\caption{The Hopf bifurcation curve (red) in the $(r,q)$-plane, for $r\in[h(\gamma_1),h(\gamma_2)]=[0.07353,12.5931]$ (where $r=I+1$). For a fixed $r$ in this range and for $q<q^\star(r)$ (below the curve), the equilibrium state $E_3(r)$ is asymptotically stable.}
\label{fig.stabreg.2d}
\end{figure}

For example, when $r=1$ (i.e. $I=0$, which corresponds to the resting state), this critical value is $q^\star(1)=0.730585$. When the fractional order $q$ of the system is lower than this critical value, the equilibrium state $E_3(1)$ is asymptotically stable. However, when the fractional order $q$ crosses this critical value, a Hopf bifurcation occurs in system (\ref{sys.2DFHR}), in a neighborhood of the equilibrium state $E_3(1)$. Numerical simulations show that this Hopf bifurcation is supercritical, i.e., it results in the appearance of a stable limit cycle. In Fig. \ref{fig.limit.cycles}, the stable limit cycles corresponding to values of $q\in\{0.75,0.8,0.85,0.9,0.95,1\}$ are shown, together with the three equilibrium states of the system ($E_1(1)$ is asymptotically stable, $E_2(1)$ is a saddle point, and $E_3(1)$ becomes unstable for $q>q^\star(1)=0.730585$.)

\begin{figure}[htbp]
\centering
\includegraphics*[width=0.8\linewidth]{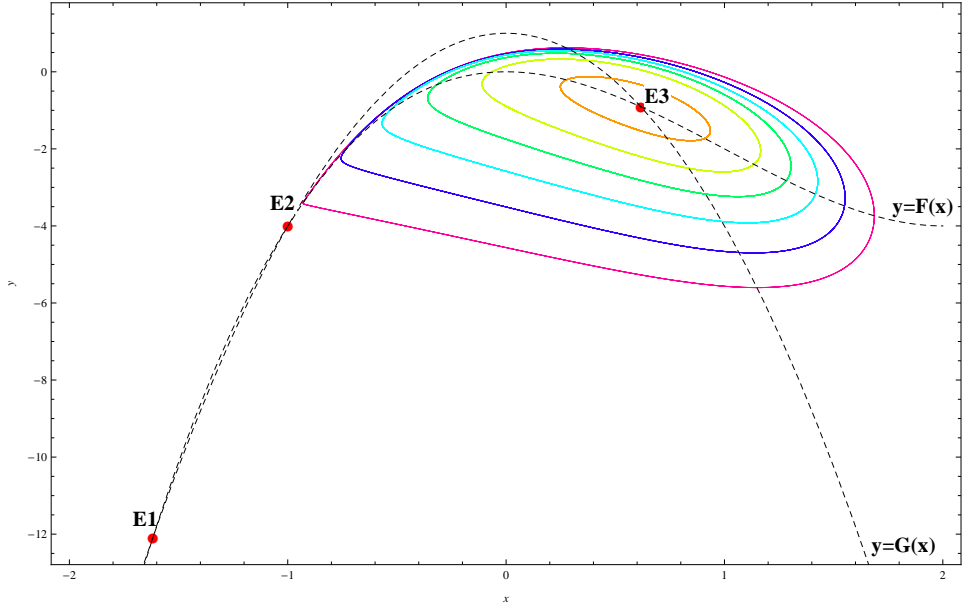}
\caption{The limit cycles corresponding to values of $q\in\{0.75,0.8,0.85,0.9,0.95,1\}$ (when $r=1$, i.e. $I=0$); the limit cycle corresponding to $q=0.75$ is the closest to the unstable equilibrium point $E_3(1)$, while the one corresponding to $q=1$ is the farthest. The equilibrium point $E_1(1)$ is asymptotically stable, and the equilibrium point $E_2(1)$ is a saddle point.}
\label{fig.limit.cycles}
\end{figure}

When $r=4.25$, (i.e. $I=3.25$), the critical value for the Hopf bifurcation is $q^\star(4.25)=0.78823$. Numerical simulations show the appearance of a stable limit cycle in a neighborhood of the equilibrium state $E_3(4.25)$, as the fractional order $q$ crosses this critical value (see Fig. \ref{fig.I325.comparison.2d}). It is worth noting that for this value of the parameter $r$, $E_3(4.25)$ is the unique equilibrium state of system (\ref{sys.2DFHR}).

\begin{figure}[htbp]
\centering
\begin{tabular}{cc}
\includegraphics[width=0.45\textwidth]{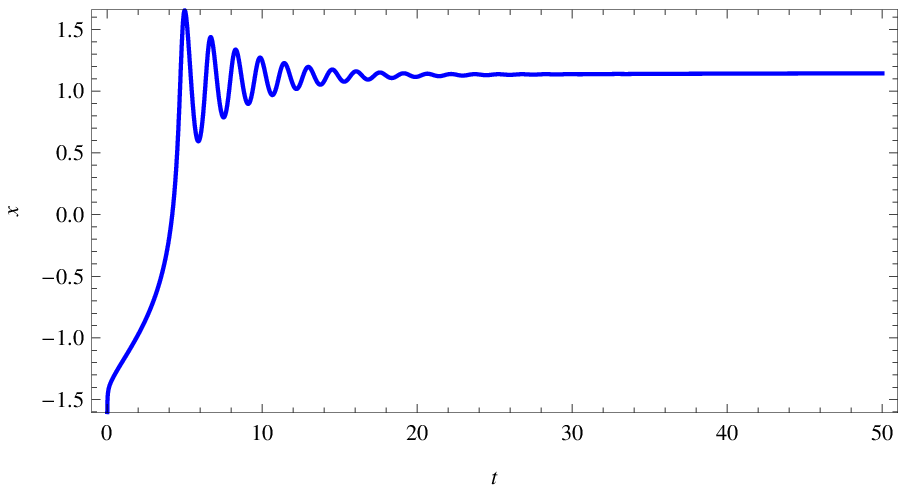} & \includegraphics[width=0.45\textwidth]{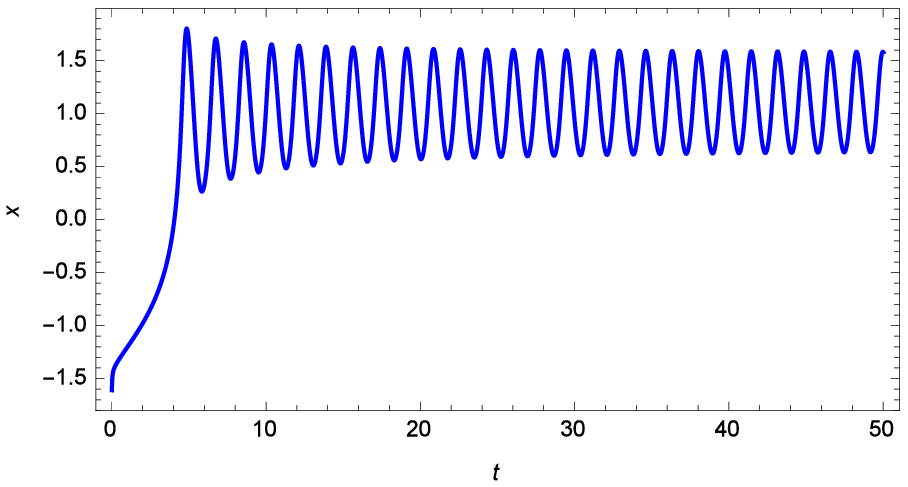}\\
\includegraphics[width=0.45\textwidth]{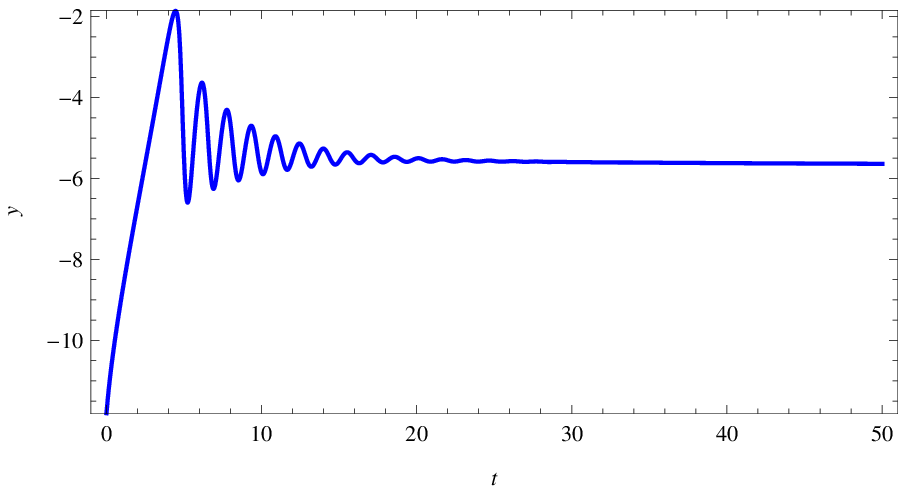} & \includegraphics[width=0.45\textwidth]{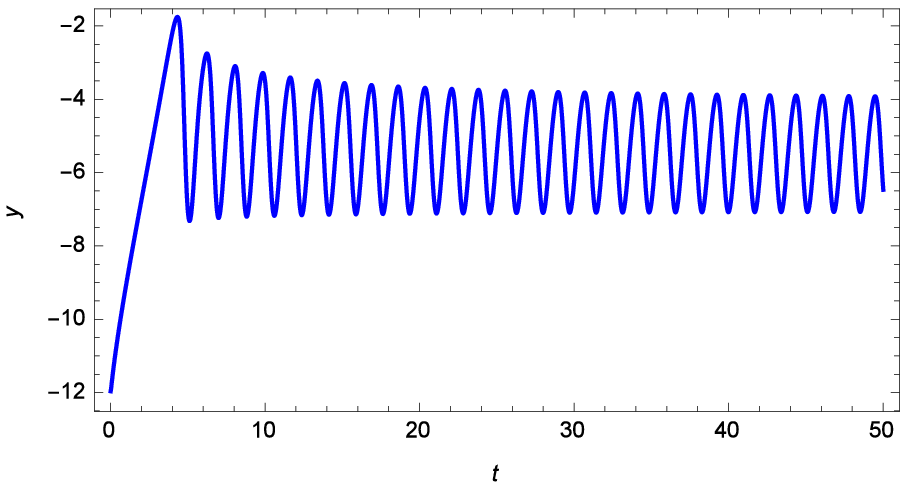}\\
(a) $q=0.75$ & (b) $q=0.8$
\end{tabular}
\caption{Trajectories of system (\ref{sys.2DFHR}) when $r=4.25$, (i.e. $I=3.25$), with initial conditions given by $E_1(1)$ (leftmost resting state corresponding to $I=0$). In the first case, (a) $q=0.75$ (shown on the left), the solution of (\ref{sys.2DFHR} converges to the asymptotically stable equilibrium state $E_3(4.25)$. In the second case, (b) $q=0.8$ (shown on the right), the solution of (\ref{sys.2DFHR} converges to the stable limit cycle which exists in the neighborhood of the unstable equilibrium state $E_3(4.25)$.}
\label{fig.I325.comparison.2d}
\end{figure}

\section{The three-dimensional fractional-order Hindmarsh-Rose neural network model}

\subsection{Model description and theoretical considerations}

Following \cite{Hindmarsh-Rose-2}, we extend the two-dimensional model (\ref{sys.2DFHR}) by adding a third equation, taking into account the slow adaptation current $z$, which will be our bursting variable. The three-dimensional slow-fast fractional-order model that we consider is:
\begin{equation}\label{sys.3DFHR}
\left\{
  \begin{array}{ll}
    ^cD^q x = y-ax^3+bx^2+I-z \\
    ^cD^q y = c-dx^2-y\\
    ^cD^q z =\varepsilon (s(x-x_0)-z)
  \end{array}
\right.
\end{equation}
where $\varepsilon\in(0,1)$ is small, $s>0$ and $x_0$ is the first coordinate of the leftmost equilibrium point of the system without adaptation (\ref{sys.2DFHR}) at the resting state ($I_0=0$, or equivalently, $\ds r=r_0=\frac{c}{a}$), i.e.
$$x_0=h_1^{-1}(r_0)<\alpha_1$$
where we have taken into account the notations from subsection 3.1. Due to this choice of $x_0$, it can be easily seen that $(x_0,G(x_0),0)$ is an equilibrium point of the system with adaptation (\ref{sys.3DFHR}) corresponding to the null external stimulus $I_0=0$, as in \cite{Hindmarsh-Rose-2}.

The equilibrium states of system (\ref{sys.3DFHR}) are given by the following algebraic system:
\begin{equation}
\left\{
  \begin{array}{ll}
    \ds x^3-px^2+\frac{s}{a}(x-x_0)=r \\
    y = c-dx^2=G(x)\\
    z =s(x-x_0)
  \end{array}
\right.
\end{equation}
where the notations  $\ds p=\frac{b-d}{a}$ and $\ds r=\frac{I+c}{a}$ are the same as in the previous section. Considering the cubic polynomial
$$H(x)= x^3-px^2+\frac{s}{a}(x-x_0)=h(x)+\frac{s}{a}(x-x_0)$$
it follows that system (\ref{sys.3DFHR}) has at most three equilibrium states, depending on the number of real roots of the equation
$$H(x)=r.$$

In the rest of this paper, we will consider that the following assumption holds (in accordance with the numerical data):
$$(A)\qquad\ds (b-d)^2< 3as$$

\begin{proposition}\label{prop.unique.eq.3d}
The function $H(x)$ is strictly increasing and there exists a unique branch of equilibrium states $\tilde{E}(r)=(x^\star(r),G(x^\star(r)),s(x^\star(r)-x_0))$, with $r\in\mathbb{R}$, for system (\ref{sys.3DFHR}), where $x^\star(r)=H^{-1}(r)$.
\end{proposition}

\begin{proof}
It can be easily seen that
$$H'(x)=3x^2-2px+\frac{s}{a}$$
and hence, due to assumption (A), we obtain that the discriminant is
$$\Delta=4\left(p^2-3\frac{s}{a}\right)=\frac{4}{a^2}\left[(b-d)^2-3as\right]<0.$$
We deduce that $H'$ is strictly positive, and the cubic polynomial $H$ is strictly increasing (and invertible) on $\mathbb{R}$, so it has a unique real root $x^\star(r)=H^{-1}(r)$.
\end{proof}

\subsection{Stability analysis}

The jacobian matrix of (\ref{sys.3DFHR}) at the equilibrium state $\tilde{E}(r)$ is
$$\tilde{A}(x^\star(r))=\left(
                       \begin{array}{ccc}
                         -F'(x^\star) & 1 & -1 \\
                         G'(x^\star) & -1 & 0 \\
                         \varepsilon s & 0 & -\varepsilon
                       \end{array}
                     \right)
$$
where $x^\star=x^\star(r)$ is given by Proposition \ref{prop.unique.eq.3d}.

The characteristic polynomial is given by:
$$P(\lambda)=\lambda^3+(F'(x^\star)+1+\varepsilon)\lambda^2+[(1+\varepsilon)F'(x^\star)-G'(x^\star)+(1+s) \varepsilon]\lambda +\varepsilon [F'(x^\star)-G'(x^\star)+s]
$$
Taking into account the notations from subsection 3.2, namely:
\begin{align*}
\tau(x)&=-F'(x)-1=-3ax^2+2bx-1\\
\delta(x)&=F'(x)-G'(x)=ah'(x)=ax(3x-2p)
\end{align*}
the characteristic polynomial can be rewritten as
$$P(\lambda)=\lambda^3+[\varepsilon-\tau(x^\star)]\lambda^2+[\delta(x^\star)-\varepsilon \tau(x^\star)+\varepsilon s]\lambda +\varepsilon [\delta(x^\star)+s]
$$
Applying the Routh-Hurwitz stability criterion, necessary and sufficient conditions for the asymptotic stability of the equilibrium state $\tilde{E}(r)$ can be obtained for the integer-order case $q=1$, and subsequently, for any fractional order $q\in(0,1)$. However, taking into consideration the large number of parameters involved in the expression of the characteristic polynomial $P(\lambda)$, it is difficult to obtain general conditions for asymptotic stability in terms of the parameter $r$, and therefore, the Routh-Hurwitz stability criterion will not be utilized in this paper.

\begin{remark}\label{rem.lambda1}
It is easy to verify that assumption (A) implies $\delta(x)+s> 0$, for any $x\in\mathbb{R}$. Therefore, as the product of the roots of the characteristic polynomial $P(\lambda)$ is $\lambda_1\lambda_2\lambda_3=-\varepsilon(\delta(x^\star)+s)<0$, we deduce that at least one root of the characteristic polynomial is negative. In fact, it is easy to evaluate
$$P(-\varepsilon)=\varepsilon(1-\varepsilon)s>0,$$
and hence, $P(\lambda)$ has at least one root in the interval $(-\infty,-\varepsilon)$.
\end{remark}

\begin{proposition}\label{prop.as.stab.3d.1}
For any $r\leq H(\alpha_1)$ or $\ds r\geq H\left(\frac{2b}{3a}\right)$, the equilibrium state $\tilde{E}(r)$ of system (\ref{sys.3DFHR}) is asymptotically stable (regardless of the fractional order $q$, or the value of the parameter $\varepsilon$).
\end{proposition}

\begin{proof}
If $r\leq H(\alpha_1)$, we have $x^\star=x^\star(r)=H^{-1}(r)\leq \alpha_1\leq 0$. On the other hand, if $\ds r\geq H\left(\frac{2b}{3a}\right)$,  we have $\ds x^\star=x^\star(r)=H^{-1}(r)\geq \frac{2b}{3a}\geq \alpha_2$. It can easily be deduced that in both cases, $\delta(x^\star)\geq 0$ and $\tau(x^\star)\leq -1$.

Hence, we have:
\begin{align*}
P(\tau(x^\star)-\varepsilon)&=-\varepsilon\tau(x^\star)^2+[\varepsilon^2+\varepsilon s+\delta(x^\star)]\tau(x^\star)+\varepsilon(1-\varepsilon)s\\
&\leq -\varepsilon\tau(x^\star)^2-[\varepsilon^2+\varepsilon s+\delta(x^\star)]+\varepsilon(1-\varepsilon)s\\
&= -\varepsilon\tau(x^\star)^2-\varepsilon^2(1+s) -\delta(x^\star)< 0
\end{align*}
Therefore, the characteristic polynomial $P(\lambda)$ will have a negative real root $\lambda_1\in (\tau(x^\star)-\varepsilon,-\varepsilon)$. The other two roots verify \begin{align*}
\lambda_2+\lambda_3&=\tau(x^\star)-\varepsilon-\lambda_1<0\\
\lambda_2\lambda_3&=-\ds\frac{\varepsilon[\delta(x^\star)+s]}{\lambda_1}>0
\end{align*}
and hence, they are in the left half-plane. We conclude that the equilibrium state $\tilde{E}(r)$ is asymptotically stable.
\end{proof}

\begin{proposition}\label{prop.as.stab.3d.2}
If $\ds r\in\left(H(\alpha_1),H\left(\frac{2b}{3a}\right)\right)$, the equilibrium state $\tilde{E}(r)$ of system (\ref{sys.3DFHR}) is asymptotically stable if and only if
\begin{equation}
\ds[\tau(x^\star)-\varepsilon-\lambda_1]\sqrt{-\lambda_1}<2\sqrt{\varepsilon[\delta(x^\star)+s]}\cos\frac{q\pi}{2}
\end{equation}
or equivalently,
\begin{equation}\label{cond.as.stab.3d}
q<\frac{2}{\pi}\arccos\left(\min\left(1,\max\left(0,\frac{[\tau(x^\star)-\varepsilon-\lambda_1]\sqrt{-\lambda_1}}{2\sqrt{\varepsilon[\delta(x^\star)+s]}}\right)\right)\right)
\end{equation}
where $\lambda_1=\lambda_1(r)\in(-\infty,-\varepsilon)$ is the smallest real root of the characteristic polynomial $P(\lambda)$ (see Remark \ref{rem.lambda1}).
\end{proposition}

\begin{proof}
Considering the smallest real root $\lambda_1=\lambda_1(r)\in(-\infty,-\varepsilon)$ of $P(\lambda)$, the other two roots of the characteristic polynomial satisfy
\begin{align*}
\lambda_2+\lambda_3&=\tau(x^\star)-\varepsilon-\lambda_1\\
\lambda_2\lambda_3&=-\ds\frac{\varepsilon[\delta(x^\star)+s]}{\lambda_1}>0
\end{align*}
Following Proposition \ref{prop.stability.2d} and Remark \ref{rem.equiv}, we deduce that $\lambda_2$ and $\lambda_3$ satisfy the asymptotic stability condition $\ds|\arg(\lambda)|>\frac{q\pi}{2}$ if and only if
$$\lambda_2+\lambda_3<2\sqrt{\lambda_2\lambda_3}\cos\frac{q\pi}{2}$$
which is equivalent to condition (\ref{cond.as.stab.3d}).\end{proof}

\subsection{Remarks on Hopf bifurcation phenomena}

Just as in the case of the two-dimensional system, the two main parameters that characterize the stability and bifurcation parameters in system (\ref{sys.3DFHR}) are $r=\ds\frac{I+c}{a}$, which is determined by the external stimulus, and the fractional order $q$ of the system.

Based on \cite{ElSaka} and the asymptotic stability results presented in the previous subsection, it may be concluded that Hopf bifurcations take place in (\ref{sys.3DFHR}) only when $\ds r\in\left(H(\alpha_1),H\left(\frac{2b}{3a}\right)\right)$ and the fractional order $q$ reaches the critical value \begin{equation}\label{eq.q.critical.3d}\ds q^\star(r)=\frac{2}{\pi}\arccos\left(\min\left(1,\max\left(0,\frac{[\tau(x^\star)-\varepsilon-\lambda_1]\sqrt{-\lambda_1}}{2\sqrt{\varepsilon[\delta(x^\star)+s]}}\right)\right)\right),
\end{equation}
excluding the cases when $q^\star(r)\in\{0,1\}$.

\subsection{Numerical example}
We consider the same values for the parameters appearing in system (\ref{sys.3DFHR}), as in the case of the fast system (\ref{sys.2DFHR}), in subsection 3.4.
$$a=1;\quad b=3;\quad c=1;\quad d=5.$$
As in \cite{Hindmarsh-Rose-2}, we consider
$$\varepsilon=0.005;\quad s=4.$$
We have $\ds p=\frac{b-d}{a}=-2<0$ and $\ds r=\frac{I+c}{a}=I+1$. Clearly, assumption (A) holds, and hence, $\tilde{E}(r)$ is the unique branch of equilibrium states of system (\ref{sys.3DFHR}).

We can also compute $\ds\alpha_1=-\frac{4}{3}$, $\ds\alpha_2=0$, $\ds H(\alpha_1)=2.32399$, $\ds H(\alpha_2)=6.47214$ $\ds\gamma_{1,2}=1\pm\frac{\sqrt{6}}{3}$, $\ds H(\gamma_1)=7.27968$, $\ds H(\gamma_2)=26.3313$, $\ds H\left(\frac{2b}{3a}\right)=H(2)=30.4721$.

From Proposition \ref{prop.as.stab.3d.1}, we deduce that $\tilde{E}(r)$ is asymptotically stable for any $r\leq H(\alpha_1)$ or $\ds r\geq H\left(\frac{2b}{3a}\right)$, or equivalently, for any $I\leq 1.32399$ or $I\geq 29.4721$, regardless of the fractional order $q$, or the value of the parameter $\varepsilon$. On the other hand, when $I\in(1.32399,29.4721)$, the equilibrium $\tilde{E}(r)$ looses its stability for certain combinations of the external stimulus $I$ and the fractional order $q$. Fig. \ref{fig.hopf.3d} shows the stability region in the $(I,q)$-plane for the equilibrium $\tilde{E}(r)$, as well as the critical values $q^\star(r)$ given by equation (\ref{eq.q.critical.3d}). Precise details about the dynamic behavior in a neighborhood of the equilibrium $\tilde{E}(r)$ of (\ref{sys.3DFHR}) are given in Table \ref{tab.3d}.

\begin{figure}[htbp]
\centering
\includegraphics*[width=0.5\linewidth]{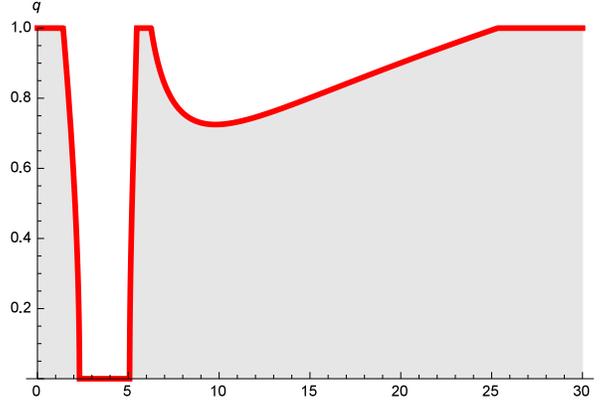}
\caption{The stability region in the $(I,q)$-plane for the 3-dimensional system, for $I\in[0,30]$. For a combination of parameters $(I,q)$ in the gray region, the equilibrium state $\tilde{E}(r)$ is asymptotically stable (with $r=I+1$). The critical values $q^\star(r)$ given by (\ref{eq.q.critical.3d}) are shown in red.}
\label{fig.hopf.3d}
\end{figure}

\begin{table}
\centering
\begin{tabular}{|l|l|}
\hline
$I\in(1.32399, 1.41401$ & $\tilde{E}(r)$  asymptotically stable for any $q\in(0,1)$ \\
\hline
$I\in (1.41401,2.31369)$ & Hopf bifurcation at $q=q^\star(r)$ \\
\hline
$I\in (2.31369,5.07454)$ & $\tilde{E}(r)$  unstable for any $q\in(0,1)$ \\
\hline
$I\in (5.07454,5.46681)$ & Hopf bifurcation at $q=q^\star(r)$ \\
\hline
$I\in (5.46681,6.25616)$ & $\tilde{E}(r)$  asymptotically stable for any $q\in(0,1)$ \\
\hline
$I\in (6.25616,25.3362)$ & Hopf bifurcation at $q=q^\star(r)$ \\
\hline
$I\in(25.3362, 29.4721$ & $\tilde{E}(r)$  asymptotically stable for any $q\in(0,1)$ \\
\hline
\end{tabular}
\caption{Dynamic behavior in a neighborhood of the equilibrium $\tilde{E}(r)$ of (\ref{sys.3DFHR}) when $I\in(1.32399,29.4721)$.}
\label{tab.3d}
\end{table}

The most interesting dynamic behavior in system (\ref{sys.3DFHR}) is observed when $I\in (2.31369,5.07454)$, i.e., when the equilibrium $\tilde{E}(r)$ is unstable for any $q\in(0,1)$. In fact, this is the range for the external stimulus where bursting behavior has been observed by numerical simulations. For $I=3.25$,  the trajectories of system (\ref{sys.3DFHR}), with initial conditions given by $\tilde{E}(1)$ (resting state corresponding to $I=0$) are shown in Fig. \ref{fig.I325.comparison.3d}, for two different values of the fractional order: $q=0.8$ and $q=0.9$. Numerical simulations suggest that as the fractional order $q$ decreases, the number of spikes in individual bursts increases.

\begin{figure}[htbp]
\centering
\begin{tabular}{cc}
\includegraphics[width=0.45\textwidth]{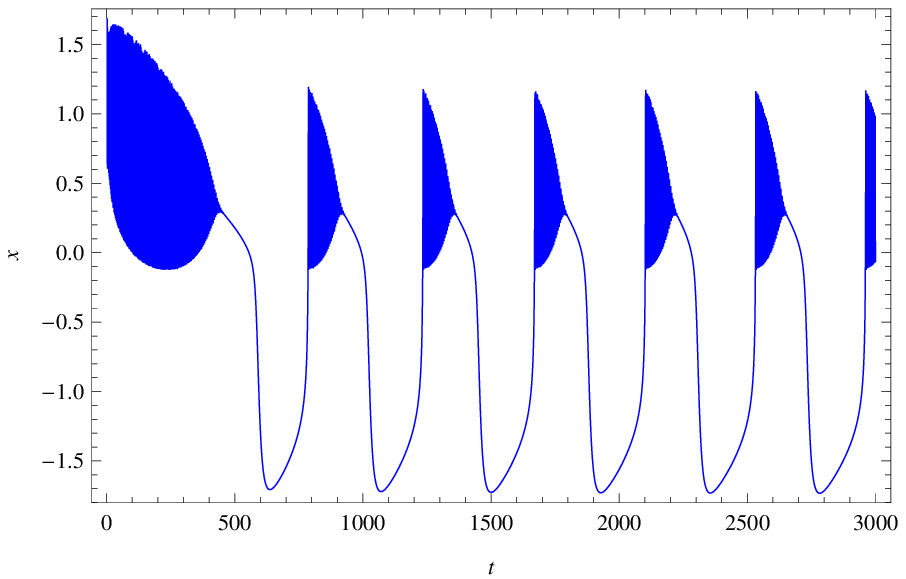} & \includegraphics[width=0.45\textwidth]{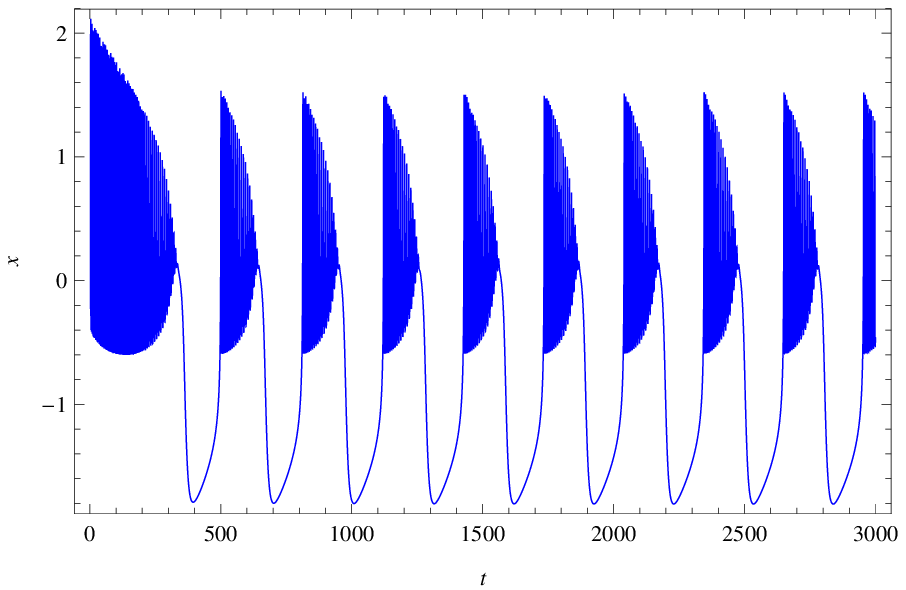} \\
\includegraphics[width=0.45\textwidth]{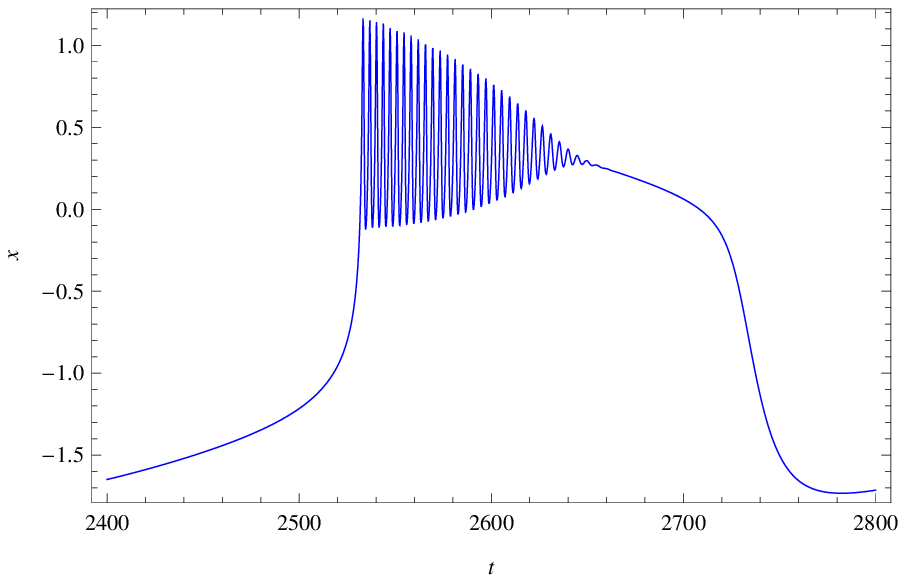} & \includegraphics[width=0.45\textwidth]{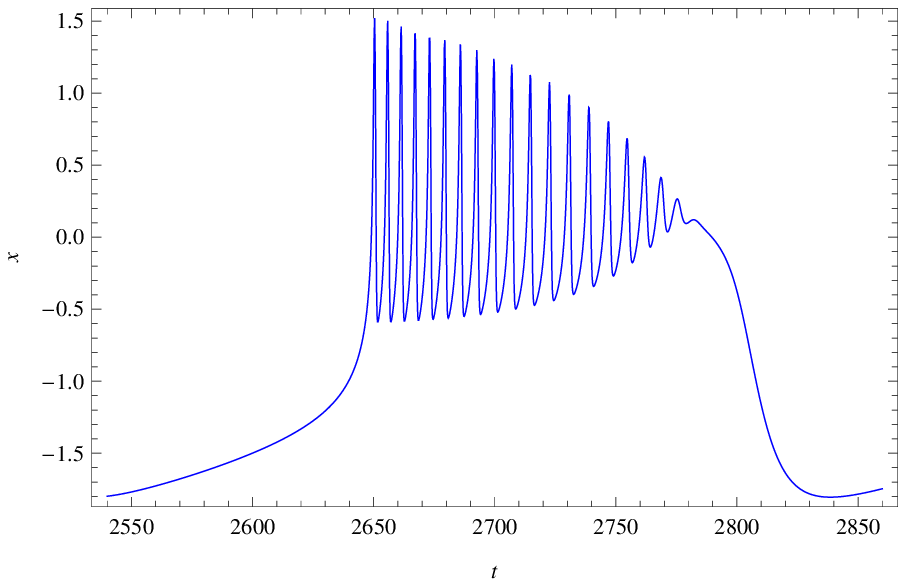}\\
\includegraphics[width=0.45\textwidth]{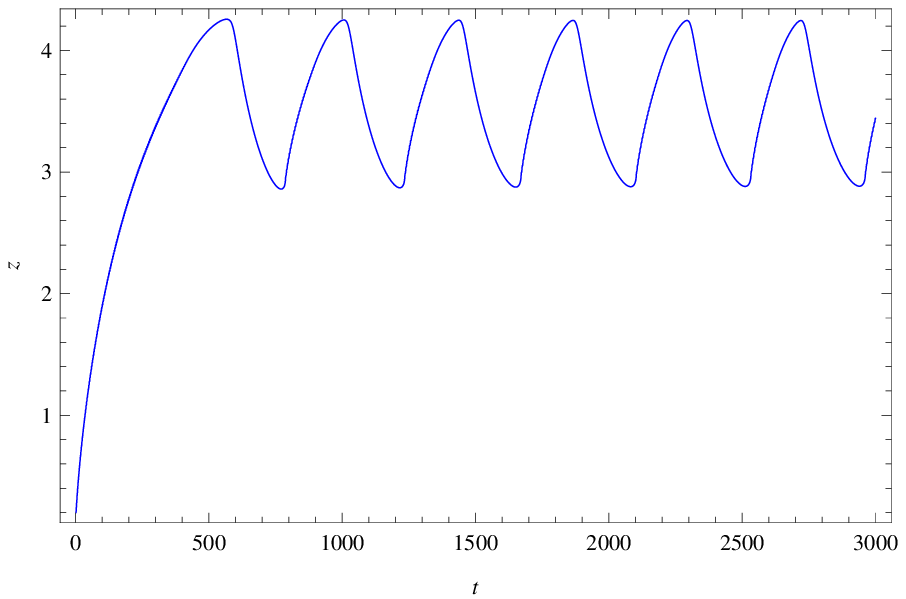} & \includegraphics[width=0.45\textwidth]{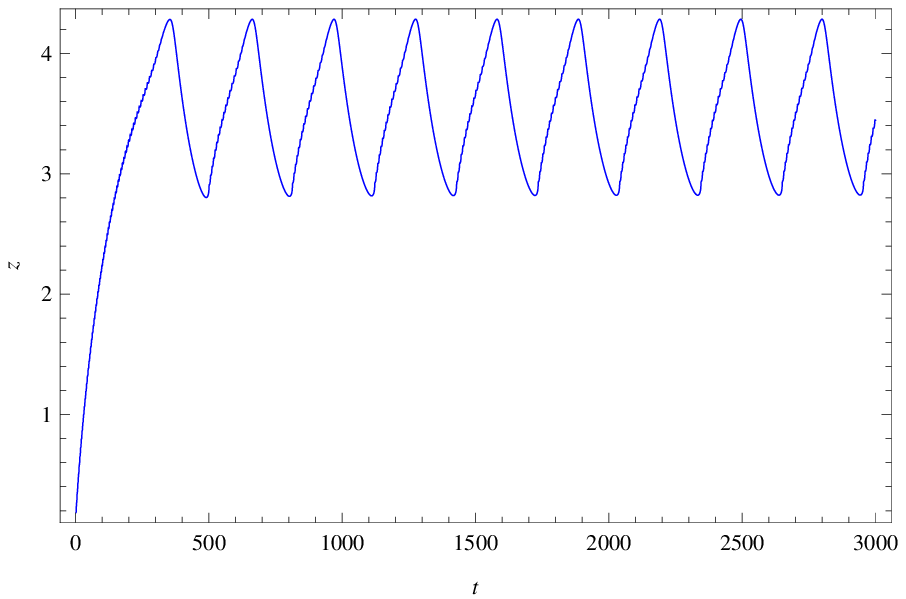} \\
\includegraphics[width=0.45\textwidth]{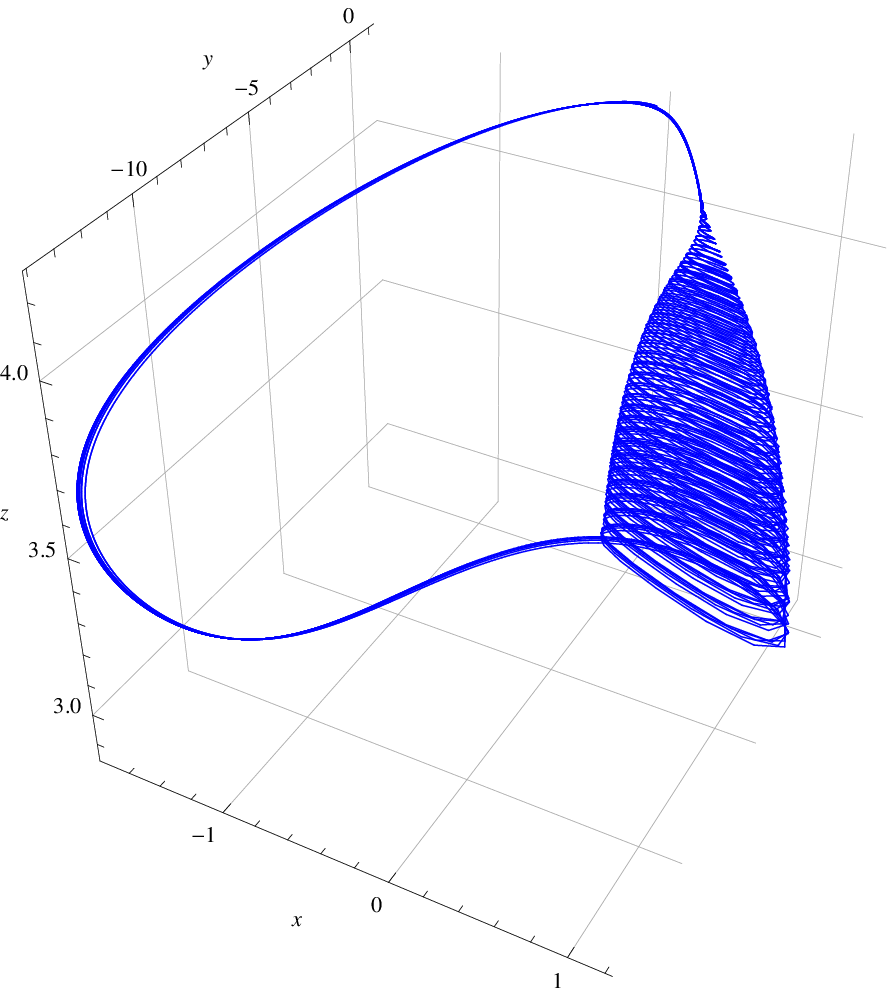} & \includegraphics[width=0.45\textwidth]{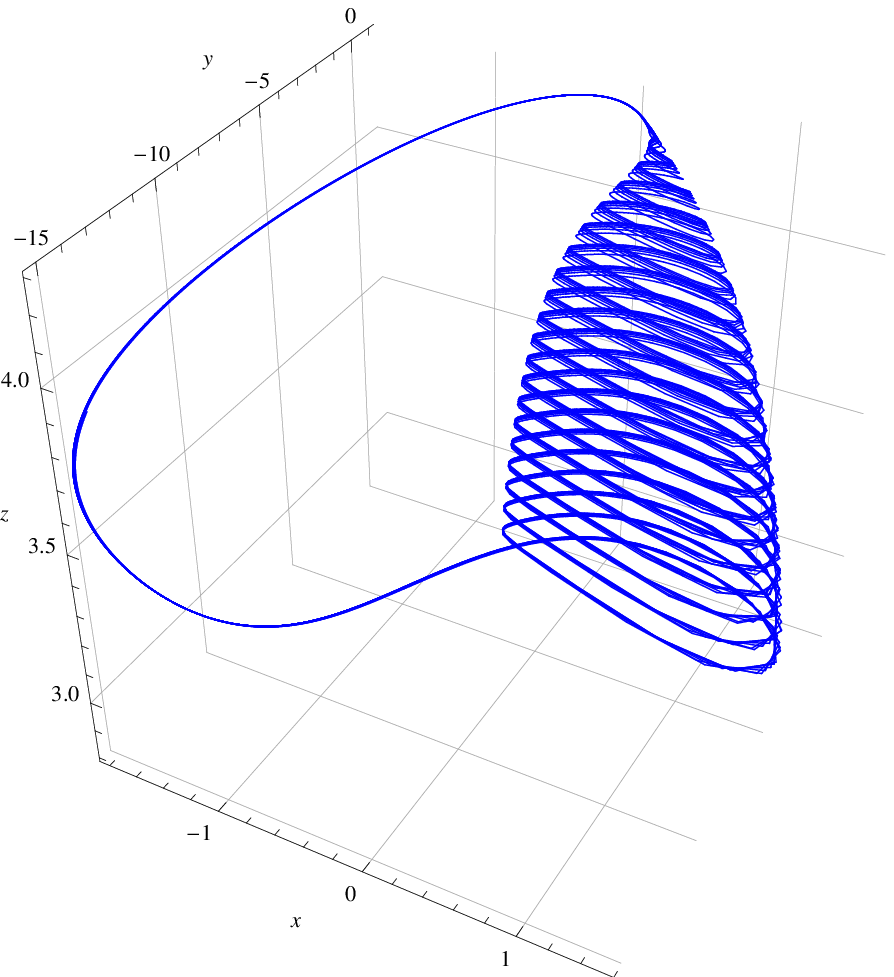} \\
(a) $q=0.8$ & (b) $q=0.9$
\end{tabular}
\caption{Comparison of trajectories of system (\ref{sys.3DFHR}), with initial conditions given by $\tilde{E}(1)$ (resting state corresponding to $I=0$) when $I=3.25$ (i.e. $r=4.25$), for two different values of the fractional order: (a) $q=0.8$ and (b) $q=0.9$.}
\label{fig.I325.comparison.3d}
\end{figure}

\section{Conclusions}

The fractional-order Hindmarsh-Rose models presented in this paper are realistic generalizations of the corresponding integer-order models, taking advantage of the fact that fractional-order derivatives are more precise in the description of dielectric processes and memory properties of membranes. Choosing the fractional order of the systems as bifurcation parameter, a theoretical stability and Hopf bifurcation analysis has been accomplished for the two- and three-dimensional fractional-order Hindmarsh-Rose models, allowing us to gain a better insight into neuronal activities. The theoretical results have been obtained without specifying fixed numerical values for the system parameters. Moreover, numerical simulations reveal rich bursting behavior in the three dimensional slow-fast model, which is consistent with experimental data. It is worth emphasizing that bursting behavior is observed for a total fractional-order of the system $3q\in(2,3)$.

\section*{Acknowledgement}

This work was supported by a grant of the Romanian National Authority for Scientific Research and Innovation, CNCS-UEFISCDI, project number PN-II-RU-TE-2014-4-0270.


\providecommand{\bysame}{\leavevmode\hbox to3em{\hrulefill}\thinspace}
\providecommand{\MR}{\relax\ifhmode\unskip\space\fi MR }
\providecommand{\MRhref}[2]{%
  \href{http://www.ams.org/mathscinet-getitem?mr=#1}{#2}
}
\providecommand{\href}[2]{#2}

\end{document}